\documentclass{amsart}

\usepackage{enumerate, amsmath, amsfonts, amssymb, amsthm, wasysym, graphics, graphicx, xcolor, url, hyperref, hypcap}
\hypersetup{colorlinks=true, citecolor=darkblue, linkcolor=darkblue}
\usepackage[all]{xy}
\usepackage{tikz}\usetikzlibrary{trees,snakes,shapes}
\graphicspath{{figures/}}

\title{The greedy flip tree of a subword complex}

\thanks{Research partially supported by grant MTM2008-04699-C03-02 and MTM2011-22792 of the spanish Ministerio de Ciencia e Innovaci\'on, by European Research Project ExploreMaps (ERC StG 208471), and by a postdoctoral grant of the Fields Institute of Toronto.}

\author{Vincent Pilaud}
\address{CNRS \& LIX, \'Ecole Polytechnique, Palaiseau}
\email{vincent.pilaud@lix.polytechnique.fr}
\urladdr{http://www.lix.polytechnique.fr/~pilaud/}

\newtheorem{theorem}{Theorem}[section]

\newtheorem{proposition}[theorem]{Proposition}
\newtheorem{lemma}[theorem]{Lemma}

\theoremstyle{definition}
\newtheorem{example}[theorem]{Example}
\newtheorem{remark}[theorem]{Remark}

\newcommand{\R}{\mathbb{R}} 
\newcommand{\Z}{\mathbb{Z}} 
\newcommand{\cX}{\mathcal{X}} 
\newcommand{\cN}{\mathcal{N}} 
\newcommand{\fS}{\mathfrak{S}} 
\newcommand{\sfr}{{\sf r}} 
\newcommand{\sfR}{{\sf R}} 

\newcommand{\set}[2]{\left\{ #1 \;\middle|\; #2 \right\}} 
\newcommand{\multiset}[2]{\left\{\!\!\left\{ #1 \;\middle|\; #2 \right\}\!\!\right\}} 
\newcommand{\ssm}{\smallsetminus} 
\newcommand{\dotprod}[2]{\langle #1 | #2 \rangle} 
\newcommand{\one}{1\!\!1} 
\newcommand{\eqdef}{\mbox{\,\raisebox{0.2ex}{\scriptsize\ensuremath{\mathrm:}}\ensuremath{=}\,}} 
\newcommand{\eraseFirst}{\vdash} 
\newcommand{\eraseLast}{\dashv} 
\newcommand{\subwordComplex}[1][\sq{Q},\rho]{\mathcal{SC}(#1)} 
\newcommand{\length}{\ell} 
\newcommand{\DemazureProduct}{\delta} 
\newcommand{\Root}[2]{{\sfr}(#1,#2)} 
\newcommand{\Roots}[1]{{\sfR}(#1)} 
\newcommand{\sq}[1]{{\rm #1}} 
\newcommand{\flipGraph}[1][\sq{Q},\rho]{\mathcal{G}(#1)} 
\newcommand{\facets}[1][\sq{Q},\rho]{\mathcal{F}(#1)} 
\newcommand{\negativeGreedy}[1][\sq{Q},\rho]{\mathsf{N}(#1)} 
\newcommand{\positiveGreedy}[1][\sq{Q},\rho]{\mathsf{P}(#1)} 
\newcommand{\negativeGreedyTree}[1][\sq{Q},\rho]{\mathcal{N}(#1)} 
\newcommand{\positiveGreedyTree}[1][\sq{Q},\rho]{\mathcal{P}(#1)} 
\newcommand{\negativeGreedyIndex}{\mathsf{n}} 
\newcommand{\positiveGreedyIndex}{\mathsf{p}} 
\newcommand{\shiftRight}[1]{#1^{\rightarrow}} 
\newcommand{\shiftLeft}[1]{#1^{\leftarrow}} 

\DeclareMathOperator{\conv}{conv} 
\DeclareMathOperator{\inv}{inv} 
\DeclareMathOperator{\join}{\star} 

\newcommand{\fref}[1]{Figure~\ref{#1}} 
\newcommand{\ie}{\textit{i.e.}~} 
\newcommand{\ordinal}{\textsuperscript{th}} 
\definecolor{darkblue}{rgb}{0,0,0.7} 
\newcommand{\darkblue}{\color{darkblue}} 
\newcommand{\defn}[1]{\emph{\darkblue #1}} 
\renewcommand{\paragraph}[1]{\medskip\noindent\textsc{#1} ---} 

\begin{document}

\begin{abstract}
We describe a canonical spanning tree of the ridge graph of a subword complex on a finite Coxeter group. It is based on properties of greedy facets in subword complexes, defined and studied in this paper. Searching this tree yields an enumeration scheme for the facets of the subword complex. This algorithm extends the greedy flip algorithm for pointed pseudotriangulations of points or convex bodies in the plane.
\end{abstract}

\vspace*{-1cm}

\maketitle

\vspace*{-.5cm}

\section{Introduction}

Subword complexes on Coxeter groups were defined and studied by A.~Knutson and E.~Miller in the context of Gr\"obner geometry in Schubert varieties~\cite{KnutsonMiller-subwordComplex,KnutsonMiller-GroebnerGeometry}. Type~$A$ spherical subword complexes can be visually interpreted using pseudoline arrangements on primitive sorting networks. These were studied by V.~Pilaud and M.~Pocchiola \cite{PilaudPocchiola} as combinatorial models for pointed pseudotriangulations of planar point sets~\cite{RoteSantosStreinu-survey} and for multitriangulations of convex polygons~\cite{PilaudSantos-multitriangulations}. These two families of geometric graphs extend in two different ways the family of triangulations of a convex polygon.

The greedy flip algorithm was initially designed to generate all pointed pseudotriangulations of a given set of points or convex bodies in general position in the plane~\cite{PocchiolaVegter, BronnimannKettnerPocchiolaSnoeying}. It was then extended in~\cite{PilaudPocchiola} to generate all pseudoline arrangements supported by a given primitive sorting network. The goal of this paper is to generalize the greedy flip algorithm to any subword complex on any finite Coxeter system. Based on combinatorial properties of greedy facets, we construct the greedy flip tree of a subword complex, which spans its ridge graph. This tree can be visited in polynomial time per node and polynomial working space to generate all facets of the subword complex. For type~$A$ spherical subword complexes, the resulting algorithm is that of~\cite{PilaudPocchiola}, although the presentation is quite different.

The paper is organized as follows. In Section~\ref{sec:subwordComplexes}, we recall some notions on finite Coxeter systems and subword complexes. Our main results appear in Section~\ref{sec:greedyFlipTree} where we define the greedy facet of a subword complex, construct the greedy flip tree, and describe the greedy flip algorithm.


\section{Subword complexes on Coxeter groups}
\label{sec:subwordComplexes}

\subsection{Coxeter systems}
\label{subsec:CoxeterSystems}

We recall some basic notions on Coxeter systems needed in this paper. More background material can be found in~\cite{Humphreys}.

Let~$V$ be an $n$-dimensional euclidean vector space. For~$v \in V \ssm 0$, we denote by~$s_v$ the reflection interchanging~$v$ and~$-v$ while fixing pointwise the orthogonal hyperplane. We consider a \defn{finite Coxeter group}~$W$ acting on~$V$, \ie a finite group generated by orthogonal reflections of~$V$. We assume without loss of generality that the intersection of all reflecting hyperplanes of~$W$ is reduced to~$0$. 

Computations in~$W$ are simplified by root systems. A \defn{root system} for~$W$ is a set~$\Phi$ of vectors stable by~$W$ and containing precisely two opposite vectors orthogonal to each reflection hyperplane of~$W$. Fix a linear functional~$f:V \to \R$ such that $f(\beta) \ne 0$ for all~$\beta \in \Phi$. It splits the root system~$\Phi$ into the set of \defn{positive roots}~$\Phi^+ \eqdef \set{\beta \in \Phi}{f(\beta)>0}$ and the set of \defn{negative roots}~$\Phi^- \eqdef -\Phi^+$. The \defn{simple roots} are the roots which lie on the extremal rays of the cone generated by~$\Phi^+$. They form a basis~$\Delta$ of the vector space~$V$. The \defn{simple reflections}~$S \eqdef \set{s_\alpha}{\alpha \in \Delta}$ generate the Coxeter group~$W$. The pair~$(W,S)$ is called a \defn{finite Coxeter system}. For~$s \in S$, we let~$\alpha_s$ be the simple root orthogonal to the reflecting hyperplane of~$s$.

The \defn{length} of an element~$w \in W$ is the length~$\ell(w)$ of the smallest expression of~$w$ as a product of the generators in~$S$. It is also known to be the cardinality of the \defn{inversion set} of~$w$, defined as the set~$\inv(w) \eqdef \Phi^+ \cap w^{-1}(\Phi^-)$ of positive roots sent to negative roots by~$w$. An expression~$w = s_1 \cdots s_p$, with $s_1, \dots, s_p \in S$, is \defn{reduced} if~$p = \ell(w)$. The \defn{Demazure product} on the Coxeter system~$(W,S)$ is the function~$\DemazureProduct$ from the words on~$S$ to~$W$ defined inductively by
$$\DemazureProduct(\varepsilon) = e \quad\text{and}\quad \DemazureProduct(\sq{Q}s) = \begin{cases} \DemazureProduct(\sq{Q})s & \text{if } \length(\DemazureProduct(\sq{Q})s) > \length(\DemazureProduct(\sq{Q})) \\ \DemazureProduct(\sq{Q}) & \text{if } \length(\DemazureProduct(\sq{Q})s) < \length(\DemazureProduct(\sq{Q})) \end{cases},$$
where $\varepsilon$ is the empty word and $e$ is the identity of~$W$.

\begin{example}[Type~$A$ --- Symmetric groups]
The symmetric group~$\fS_{n+1}$, acting on the linear hyperplane $\one^\perp \eqdef \set{x \in \R^{n+1}}{\dotprod{\one}{x} = 0}$ by permutation of the coordinates, is the reflection group of \defn{type~$A_n$}. It is the group of isometries of the standard $n$-dimensional regular simplex $\conv \{e_1,\dots,e_{n+1}\}$. See \fref{fig:coxeterArrangement} (left). Its reflections are the transpositions of~$\fS_{n+1}$ and the set~$\set{e_i-e_j}{i,j \in [n+1]}$ is a root system for~$A_n$. We can choose the linear functional~$f$ such that the simple reflections are the adjacent transpositions~$\tau_i \eqdef (i\;\;i+1)$, for~$i \in [n]$, and the simple roots are the vectors~$e_{i+1}-e_i$, for~$i \in [n]$.
\end{example}

\begin{figure}[b]
	\centerline{\includegraphics[width=\textwidth]{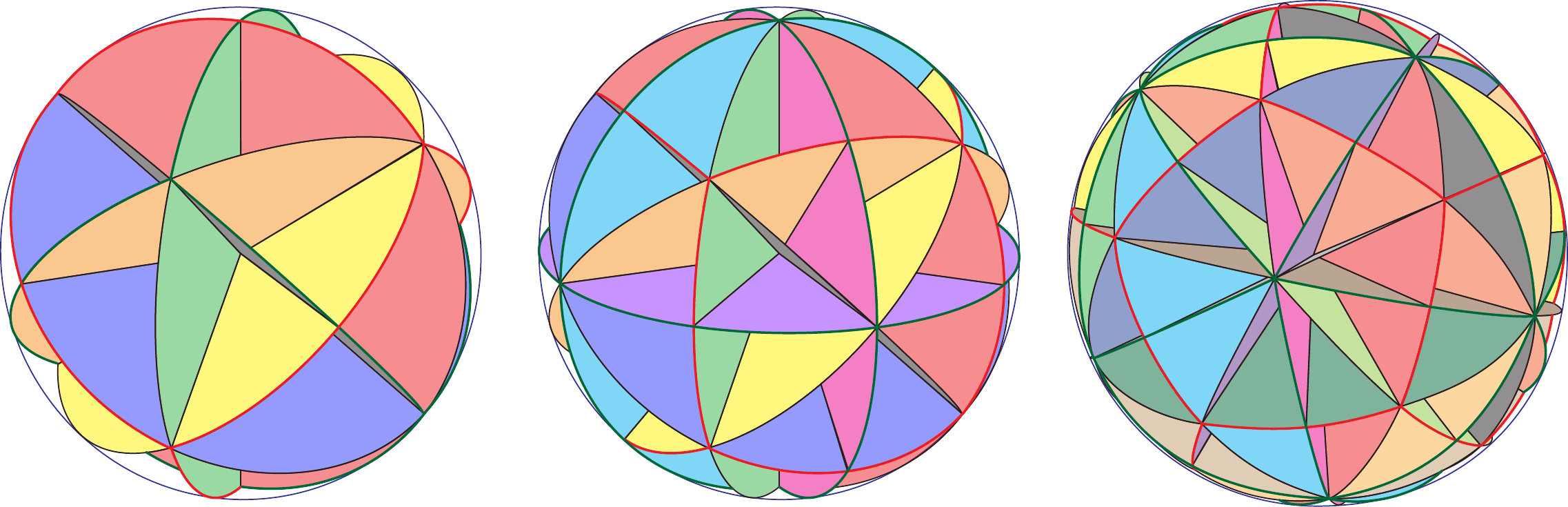}}
	\caption{The $A_3$-, $B_3$-, and $H_3$-arrangements.}
	\label{fig:coxeterArrangement}
\end{figure}

\begin{example}[Type~$B$ --- Hyperoctahedral groups]
The semidirect product of the symmetry group of~$\fS_n$ (acting on~$\R^n$ by permutation of the coordinates) with the group~$(\Z_2)^n$ (acting on~$\R^n$ by sign change) is the reflection group of \defn{type~$B_n$}. It is the isometry group of the \mbox{$n$-dimen}\-sional regular cross-polytope~$\conv \{\pm e_1,\dots,\pm e_n\}$ and of its polar $n$-dimensional regular cube~$[-1,1]^n$. See \fref{fig:coxeterArrangement} (middle). Its reflections are the transpositions of~$\fS_n$ and the changes of one single sign. The set $\set{\pm e_p \pm e_q}{p<q \in [n]} \cup \set{\pm e_p}{p \in [n]}$ is a root system for~$B_n$. We can choose the linear functional~$f$ such that the simple reflections are the adjacent transpositions $\tau_i \eqdef (i\;\;i+1)$, for~$i \in [n-1]$, together with the change~$\chi$ of the first sign, and thus the simple roots are the vectors $e_{i+1}-e_i$, for~$i \in [n-1]$, together with~the~vector~$e_1$.
\end{example}

\begin{example}[Type~$H_3$ --- Icosahedral group]
The isometry group of the regular icosahedron (and of its polar dodecahedron) is a Coxeter group. See \fref{fig:coxeterArrangement}~(right).
\end{example}


\subsection{The subword complex}
\label{subsec:subwordComplex}

Consider a finite Coxeter system~$(W,S)$, a word $\sq{Q} \eqdef q_1q_2 \cdots q_m$ on the generators of~$S$, and an element~$\rho \in W$. A.~Knutson and E.~Miller~\cite{KnutsonMiller-subwordComplex} define the \defn{subword complex}~$\subwordComplex$ to be the simplicial complex of subwords of~$\sq{Q}$ whose complements contain a reduced expression for~$\rho$ as a subword. A vertex of~$\subwordComplex$ is a position in~$\sq{Q}$. We denote by~$[m] \eqdef \{1,2,\dots,m\}$ the set of positions in~$\sq{Q}$. A facet of~$\subwordComplex$ is the complement of a set of positions which forms a reduced expression for~$\rho$ in~$\sq{Q}$. We denote by~$\facets$ the set of facets of~$\subwordComplex$. We write~$\rho \prec \sq{Q}$ when~$\sq{Q}$ contains a reduced expression of~$\rho$, \ie when~$\subwordComplex$ is non-empty.

\begin{example}
\label{exm:toto}
Consider the type~$A$ Coxeter group~$\fS_4$ generated by~$\set{\tau_i}{i \in [3]}$. Let~$\bar{\sq{Q}} \eqdef \tau_2\tau_3\tau_1\tau_3\tau_2\tau_1\tau_2\tau_3\tau_1$ and~$\bar\rho \eqdef [4,1,3,2]$. The reduced expressions of~$\bar\rho$ are $\tau_2\tau_3\tau_2\tau_1$, $\tau_3\tau_2\tau_3\tau_1$, and $\tau_3\tau_2\tau_1\tau_3$. Thus, the facets of the subword complex~$\subwordComplex[\bar{\sq{Q}},\bar\rho]$ are $\{1, 2, 3, 5, 6\}$, $\{1, 2, 3, 6, 7\}$, $\{1, 2, 3, 7, 9\}$, $\{1, 3, 4, 5, 6\}$, $\{1, 3, 4, 6, 7\}$, $\{1, 3, 4, 7, 9\}$, $\{2, 3, 5, 6, 8\}$, $\{2, 3, 6, 7, 8\}$, $\{2, 3, 7, 8, 9\}$, $\{3, 4, 5, 6, 8\}$, $\{3, 4, 6, 7, 8\}$, and $\{3, 4, 7, 8, 9\}$.
We denote by~$\bar I \eqdef \{1,3,4,7,9\}$ and~$\bar J \eqdef \{3,4,7,8,9\}$. We will use this example as a recurrent example in this paper to illustrate further notions.
\end{example}

\begin{example}[Type~$A$ --- Primitive networks]
\label{exm:typeA}

\begin{figure}[b]
	\centerline{\includegraphics[width=.9\textwidth]{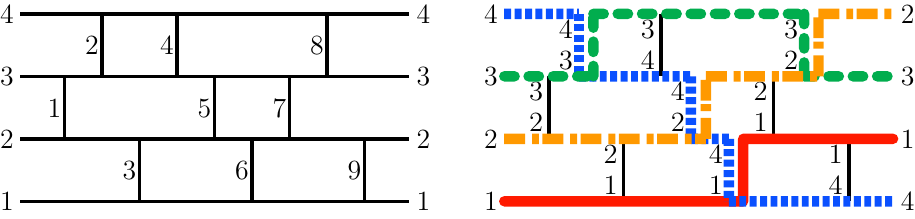}}
	\caption{The network~$\cN_{\bar{\sq{Q}}}$ (left) and the pseudoline arrangement~$\Lambda_{\bar I}$ for the facet~$\bar I = \{1,3,4,7,9\}$ of~$\subwordComplex[\bar{\sq{Q}},\bar\rho]$ (right).}
	\label{fig:network}
\end{figure}

For type~$A$ Coxeter systems, subword complexes can be visually interpreted using primitive networks. A \defn{network}~$\cN$ is a collection of~$n+1$ horizontal lines (called \defn{levels}, and labeled from bottom to top), together with~$m$ vertical segments (called \defn{commutators}, and labelled from left to right) joining two different levels and such that no two of them have a common endpoint. We only consider \defn{primitive} networks, where any commutator joins two consecutive levels. See \fref{fig:network} (left).
A \defn{pseudoline} supported by the network~$\cN$ is an abscissa monotone path on~$\cN$. A commutator of~$\cN$ is a \defn{crossing} between two pseudolines if it is traversed by both pseudolines, and a \defn{contact} if its endpoints are contained one in each pseudoline. A \defn{pseudoline arrangement}~$\Lambda$ is a set of~$n+1$ pseudolines on~$\cN$, any two of which have at most one crossing, possibly some contacts, and no other intersection. We label the pseudolines of~$\Lambda$ from bottom to top on the left of the network, and we define~$\pi(\Lambda) \in \fS_{n+1}$ to be the permutation given by the order of these pseudolines on the right of the network. Note that the crossings of~$\Lambda$ correspond to the inversions of~$\pi(\Lambda)$. See \fref{fig:network} (right).

Consider the type~$A$ Coxeter group~$\fS_{n+1}$ generated by~$S = \set{\tau_i}{i \in [n]}$, where~$\tau_i$ is the adjacent transposition~$(i\;\;i+1)$. To a word~$\sq{Q} \eqdef q_1q_2 \cdots q_m$ with~$m$ letters on~$S$, we associate a primitive network~$\cN_\sq{Q}$ with~$n+1$ levels and~$m$ commutators. If~$q_j = \tau_p$, the $j$\ordinal{} commutator of~$\cN_\sq{Q}$ is located between the $p$\ordinal{} and $(p+1)$\ordinal{} levels of~$\cN_\sq{Q}$. See \fref{fig:network} (left). For~$\rho \in \fS_{n+1}$, a facet~$I$  of~$\subwordComplex$ corresponds to a pseudoline arrangement~$\Lambda_I$ supported by~$\cN_\sq{Q}$ and with~$\pi(\Lambda_I) = \rho$. The positions of the contacts (resp.~crossings) of~$\Lambda_I$ correspond to the elements of~$I$ (resp.~of the complement of~$I$). See \fref{fig:network} (right).
\end{example}

\begin{example}[Combinatorial models for geometric graphs]
\label{exm:geometricGraphs}
As pointed out in~\cite{PilaudPocchiola}, pseudoline arrangements on primitive networks give combinatorial models for the following families of geometric graphs (see \fref{fig:geometricGraphs}):
\begin{enumerate}[(i)]
\item triangulations of convex polygons;
\item multitriangulations of convex polygons~\cite{PilaudSantos-multitriangulations};
\item pointed pseudotriangulations of points in general position in the plane~\cite{RoteSantosStreinu-survey};
\item pseudotriangulations of disjoint convex bodies in the plane~\cite{PocchiolaVegter}.
\end{enumerate}
For example, consider a triangulation~$T$ of a convex~$(n+3)$-gon. Define the direction of a line of the plane to be the angle~$\theta \in [0,\pi)$ of this line with the horizontal axis. Define also a bisector of a triangle~$\triangle$ to be a line passing through a vertex of~$\triangle$ and separating the other two vertices of~$\triangle$. For any direction~$\theta \in [0,\pi)$, each triangle of~$T$ has precisely one bisector in direction~$\theta$. We can thus order the~$n+1$ triangles of~$T$ according to the order~$\pi_\theta$ of their bisectors in direction~$\theta$. The pseudoline arrangement associated to~$T$ is then given by the evolution of the order~$\pi_\theta$ when the direction~$\theta$ describes the interval~$[0,\pi)$. A similar duality holds for the other three families of graphs, replacing triangles by the natural cells decomposing the geometric graph (stars for multitriangulations~\cite{PilaudSantos-multitriangulations}, or pseudotriangles for pseudotriangulations~\cite{RoteSantosStreinu-survey}). See \fref{fig:geometricGraphs} for an illustration. Details can be found in~\cite{PilaudPocchiola}.

\begin{figure}[h]
	\centerline{\includegraphics[width=\textwidth]{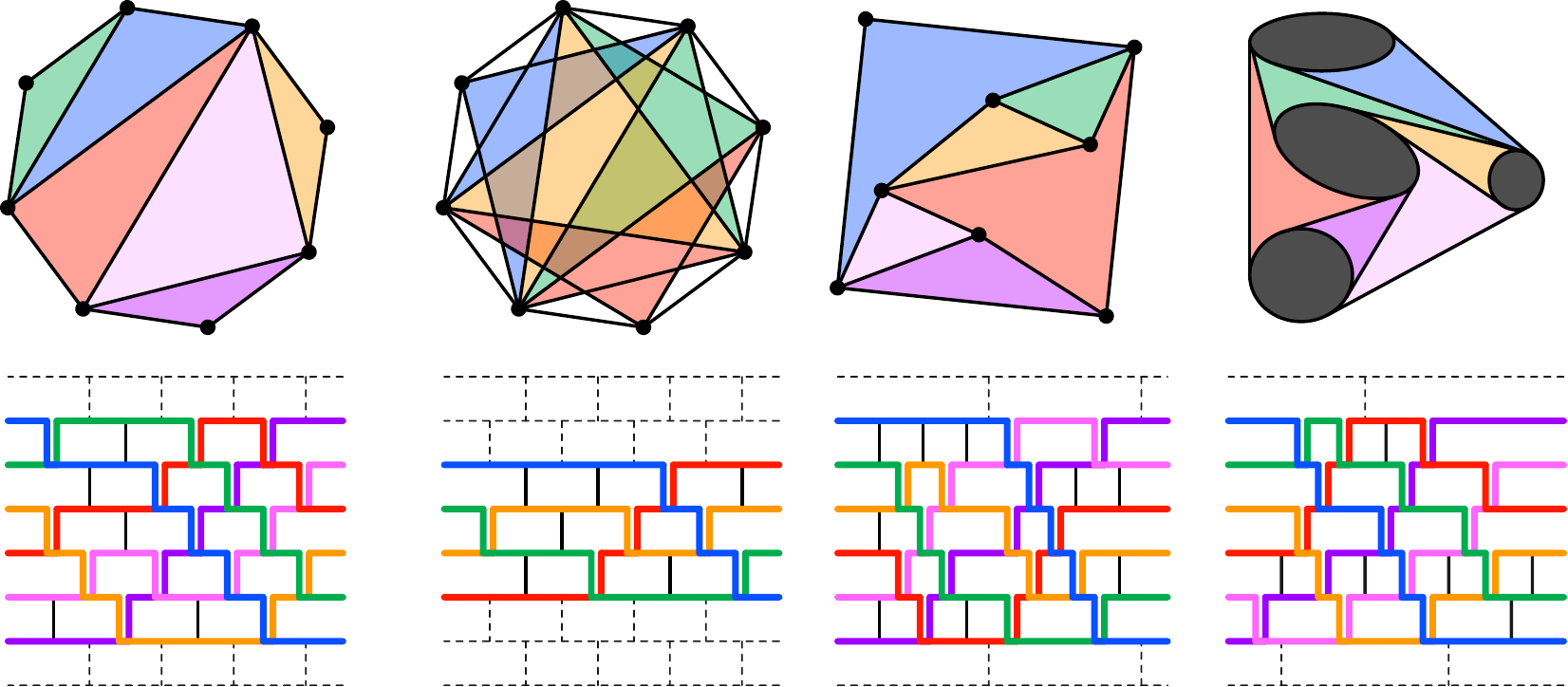}}
	\caption{Primitive sorting networks are combinatorial models for certain families of geometric graphs.}
	\label{fig:geometricGraphs}
\end{figure}
\end{example}

\begin{example}[Type~$B$ --- Symmetric primitive networks]
Consider the type~$B$ Coxeter group~$\fS_n \rtimes (\Z_2)^n$ acting on~$\R^n$ and generated by~${S = \set{\tau_i}{i \in [n-1]} \cup \chi}$, where~$\tau_i$ exchange the $i$\ordinal{} and~$(i+1)$\ordinal{} coordinates, and~$\chi$ changes the sign of the first coordinate. To a word~$\sq{Q} \eqdef q_1q_2 \cdots q_m$ on~$S$ with~$m$ letters and~$x$ occurrences of~$\chi$, we associate a primitive network~$\cN_\sq{Q}$ with~$2n$ levels and~$2m-x$ commutators, which is symmetric with respect to the horizontal axis. The levels of~$\cN_\sq{Q}$ are labeled by $-n,\dots,-1,1,\dots,n$ from bottom to top. An occurrence of~$\tau_i$ is replaced by a pair of symmetric commutators between~$-i-1$ and~$-i$ and between~$i$ and~$i+1$, and an occurrence of~$\chi$ is replaced by a commutator between~$-1$ and~$1$. A facet~$I$ of the subword complex~$\subwordComplex$ is represented by a symmetric pseudoline arrangement~$\Lambda_I$ supported by~$\cN_\sq{Q}$ whose contacts correspond to the positions in~$I$. If the pseudolines of~$\Lambda_I$ are labeled by~$-n,\dots,-1,1,\dots,n$ from bottom to top on the left of~$\cN_\sq{Q}$, then their order on the right of~$\cN_\sq{Q}$ is given by $-\rho(n),\dots,-\rho(1),\rho(1),\dots,\rho(n)$.
\end{example}

\begin{example}[Combinatorial models for centrally symmetric geometric graphs]
Type~$B$ subword complexes provide combinatorial models for the centrally symmetric versions of the geometric graphs of Example~\ref{exm:geometricGraphs}. Indeed, the central symmetry of a geometric graph translates into an horizontal symmetry on its dual pseudoline arrangement. See also the discussion in~\cite{CeballosLabbeStump} in particular the dictionnary in Table~2.
\end{example}


\subsection{Generating the subword complex}
\label{subsec:generationgSubwordComplex}

In this paper, we discuss the problem to exhaustively generate the set~$\facets$ of facets of the subword complex~$\subwordComplex$. We underline in this section two immediate enumeration algorithms which illustrate relevant properties of the subword complex.

For the evaluation of the time and space complexity of the different enumeration algorithms, we consider as parameters the rank~$n$ of the Coxeter group~$W$, the size~$m$ of the word~$\sq{Q}$, and the length~$\ell$ of the element~$\rho$. None of these parameters can be considered to be constant a priori. For example, if we want to generate all triangulations of a convex $(n+3)$-gon (see Example~\ref{exm:geometricGraphs}), we consider a subword complex with a group~$W$ of rank~$n$, a word~$\sq{Q}$ of size~${n(n+3)/2}$, and an element~$\rho$ of length~${n(n+1)/2}$.

\paragraph{Inductive structure}
The first method to generate~$\facets$ relies on the inductive structure of the family of subword complexes. Throughout this paper, we denote by~$\sq{Q}_\eraseFirst \eqdef q_2 \cdots q_m$ and~$\sq{Q}_\eraseLast \eqdef q_1 \cdots q_{m-1}$ the words on~$S$ obtained from~$\sq{Q} \eqdef q_1 \cdots q_m$ by deleting its first and last letters respectively. For a set~$\cX$ of subsets of~$\Z$, we denote by~$\cX \join z \eqdef z \join \cX \eqdef \set{X \cup z}{X \in \cX}$ the join of~$\cX$ with some~$z \in \Z$. Moreover, let~$\shiftRight{\cX} \eqdef \set{\shiftRight{X}}{X \in \cX}$, where~$\shiftRight{X} \eqdef \set{x+1}{x \in X}$ denotes the right shift of the set~$X \in \cX$. Remember that $\ell(\rho)$ denotes the length of~$\rho$ and that we write~$\rho \prec \sq{Q}$ when~$\sq{Q}$ contains a reduced expression of~$\rho$.

We can decompose inductively the facets of~$\facets$ according on whether or not they contain the last letter of~$\sq{Q}$:
\begin{equation}
\label{eq:inductionRight}
\facets =
\begin{cases}
	\facets[\sq{Q}_\eraseLast,\rho q_m] & \text{if } \rho \not\prec \sq{Q}_\eraseLast, \\
	\facets[\sq{Q}_\eraseLast, \rho] \join m & \text{if } \ell(\rho q_m) > \ell(\rho), \\
	\facets[\sq{Q}_\eraseLast,\rho q_m] \, \sqcup \, \big( \facets[\sq{Q}_\eraseLast,\rho] \join m \big) & \text{otherwise}.
\end{cases}
\end{equation}
For later reference, let us also explicitly write the inductive decomposition of the facets of~$\facets$ according on whether or not they contain the first letter of~$\sq{Q}$:

\begin{equation}
\label{eq:inductionLeft}
\facets =
\begin{cases}
	\shiftRight{\facets[\sq{Q}_\eraseFirst,q_1\rho]} & \text{if } \rho \not\prec \sq{Q}_\eraseFirst \\
	1 \join \shiftRight{\facets[\sq{Q}_\eraseFirst,\rho]} & \text{if } \ell(q_1\rho) > \ell(\rho), \\
	\shiftRight{\facets[\sq{Q}_\eraseFirst,q_1\rho]} \, \sqcup \, \big( 1 \join \shiftRight{\facets[\sq{Q}_\eraseFirst,\rho]} \big) & \text{otherwise}.
\end{cases}
\end{equation}

The inductive structure of~$\subwordComplex$ yields an inductive algorithm for the enumeration of~$\facets$, whose running time per facet is polynomial. More precisely, since all subword complexes which appear in the different cases of the induction formula~(\ref{eq:inductionRight}) are non-empty, and since the tests~$\rho \not\prec \sq{Q}_\eraseLast$ and~$\ell(\rho q_m) > \ell(\rho)$ can be performed in~$O(mn)$ time, the running time per facet of this inductive algorithm is in~$O(m^2n)$.

The inductive structure of~$\subwordComplex$ is moreover useful for the following result.
\begin{theorem}[\cite{KnutsonMiller-subwordComplex}]
\label{theo:KnutsonMiller}
The subword complex~$\subwordComplex$ is a topological sphere if~$\rho$ is precisely the Demazure product~$\DemazureProduct(\sq{Q})$ of~$\sq{Q}$, and a topological ball otherwise.
\end{theorem}

\paragraph{The flip graph}
The second direct method to generate~$\facets$ relies on flips. Let~$I$ be a facet of~$\subwordComplex$ and~$i$ be an element of~$I$. If there exists a facet~$J$ of $\subwordComplex$ and an element~$j \in J$ such that~$I \ssm i = J \ssm j$, we say that $I$~and~$J$ are \defn{adjacent} facets, that~$i$ is \defn{flippable} in~$I$, and that~$J$ is obtained from~$I$ by \defn{flipping}~$i$. Note that, if they exist,~$J$ and~$j$ are unique by Theorem~\ref{theo:KnutsonMiller}. We denote by~$\flipGraph$ the graph of flips, whose vertices are the facets of~$\subwordComplex$ and whose edges are pairs of adjacent facets. In other words, $\flipGraph$~is the ridge graph of the simplicial complex~$\subwordComplex$.

\begin{example}
\fref{fig:flipGraph} represents the flip graph~$\flipGraph[\bar{\sq{Q}},\bar\rho]$ for the subword complex~$\subwordComplex[\bar{\sq{Q}}, \bar\rho]$ of Example~\ref{exm:toto}.
\begin{figure}[h]
	\centerline{\includegraphics[width=\textwidth]{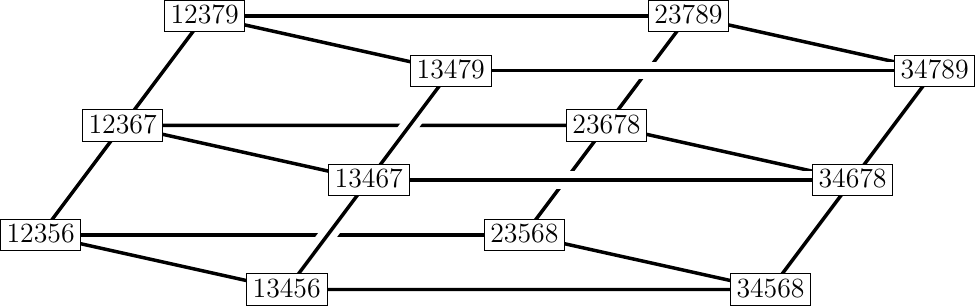}}
	\caption{The flip graph~$\flipGraph[\bar{\sq{Q}},\bar\rho]$. Facets of~$\subwordComplex[\bar{\sq{Q}},\bar\rho]$ appear in lexicographic order from left to right.}
	\label{fig:flipGraph}
\end{figure}
\end{example}

Since the flip graph~$\flipGraph$ is connected by Theorem~\ref{theo:KnutsonMiller}, we can explore it to generate~$\facets$. Since~$\flipGraph$ has degree bounded by~$m-\ell(\rho)$, we need~${O(m-\ell(\rho))}$ flips per facet for this exploration. However, we need to store all facets of~$\subwordComplex$ during the algorithm, which may require an exponential working space. This happens for example if we want to generate the~$\frac{1}{n+2}{2n+2 \choose n+1}$ triangulations of a convex $(n+3)$-gon (see Example~\ref{exm:geometricGraphs}).

\medskip
In this paper, we present the \emph{greedy flip algorithm} to generate the facets of the subword complex~$\subwordComplex$. This algorithm explores a spanning tree of the graph of flips~$\flipGraph$, which we call the \emph{greedy flip tree}. The construction of this tree is based on \emph{greedy facets} in subword complexes and on their inductive structure, similar to the inductive structure of the subword complexes described above. The running time per facet of the greedy flip algorithm is also in~$O(m^2n)$, while its working space is in~$O(mn)$. We compare experimental running times of the inductive algorithm and of the greedy flip algorithm later in Section~\ref{subsec:greedyFlipTree}.


\subsection{Roots and flips}
\label{subsec:roots&flips}

Throughout the paper, we consider a flip in a subword complex as an elementary operation to measure the complexity of our algorithm. In practice, the necessary information to perform flips in a facet~$I$ of~$\subwordComplex$ is encoded in its root function~$\Root{I}{\cdot} : [m] \to \Phi$ defined by
$$\Root{I}{k} \eqdef \sigma_{[k-1] \ssm I}(\alpha_{q_k}),$$
where~$\sigma_X$ denotes the product of the reflections~$q_x$ for~$x \in X$. The \defn{root configuration} of the facet~$I$ is the multiset~$\Roots{I} \eqdef \multiset{\Root{I}{i}}{i \text{ flippable in } I}$. The root function was introduced by C.~Ceballos, J.-P.~Labb\'e and C.~Stump~\cite{CeballosLabbeStump} and the root configuration was extensively studied by V.~Pilaud and C.~Stump~\cite{PilaudStump} in the construction of brick polytopes for spherical subword complexes. The main properties of the root function are summarized in the following proposition, whose proof is similar to that of~\cite[Lemmas~3.3 and 3.6]{CeballosLabbeStump} or~\cite[Lemma~3.3]{PilaudStump}.

\begin{proposition}
\label{prop:roots&flips}
Let~$I$ be any facet of the subword complex~$\subwordComplex$.
\begin{enumerate}
\item
\label{prop:roots&flips:inversions}
The map~$\Root{I}{\cdot}:i \mapsto \Root{I}{i}$ is a bijection from the complement of~$I$ to the inversion set of~$\rho^{-1}$.

\item
\label{prop:roots&flips:flippable}
The map~$\Root{I}{\cdot}$ sends the flippable elements in~$I$ to~$\set{\pm \beta}{\beta \in \inv(\rho^{-1})}$ and the unflippable ones to~$\Phi^+ \ssm \inv(\rho^{-1})$.

\item
\label{prop:roots&flips:flip}
If $I$~and~$J$ are two adjacent facets of~$\subwordComplex$ with~$I \ssm i = J \ssm j$, the position $j$ is the unique position in the complement of~$I$ for which~${\Root{I}{j} = \pm\Root{I}{i}}$. Moreover, $\Root{I}{i} = \Root{I}{j} \in \Phi^+$ when $i < j$, while $\Root{I}{i} = -\Root{I}{j} \in \Phi^-$ when $j < i$.

\item
\label{prop:roots&flips:update}
In the situation of (\ref{prop:roots&flips:flip}), the map~$\Root{J}{\cdot}$ is obtained from the map~$\Root{I}{\cdot}$ by:
$$\Root{J}{k} = \begin{cases} s_{\Root{I}{i}}(\Root{I}{k}) & \text{if } \min(i,j) < k \le \max(i,j) \\ \Root{I}{k} & \text{otherwise.} \end{cases}$$
\end{enumerate}
\end{proposition}

Observe that this proposition ensures in particular that we can perform flips in the subword complex~$\subwordComplex$ in~$O(mn)$ time if we store and update the facets of~$\subwordComplex$ together with their root functions. Note that this storage requires~$O(mn)$ space.

\begin{figure}[b]
	\centerline{\includegraphics[width=.9\textwidth]{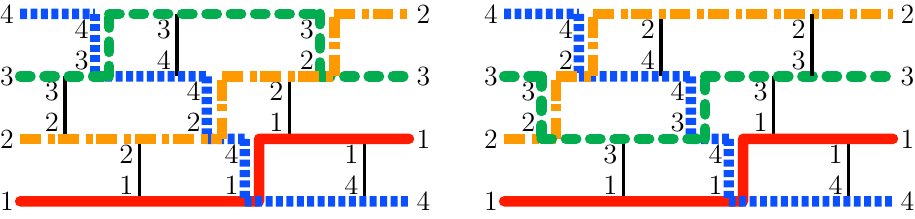}}
	\caption{The flip between the adjacent facets~$\bar I = \{1,3,4,7,9\}$ and $\bar J = \{3,4,7,8,9\}$ of~$\subwordComplex[\bar{\sq{Q}},\bar\rho]$, illustrated on the network~$\cN_{\bar{\sq{Q}}}$.}
	\label{fig:flip}
\end{figure}

\begin{example}
In type~$A$ (and~$B$), roots and flips are easily read on the primitive network interpretation presented in Example~\ref{exm:typeA}. Consider a word~$\sq{Q}$ on~$\set{\tau_i}{i \in [n]}$, an element~$\rho \in \fS_{n+1}$, and a facet~$I$ of~$\subwordComplex$. For any~$k \in [m]$, the root~$\Root{I}{k}$ is the difference~$e_t-e_b$ where~$t$ and~$b$ are the indices of the pseudolines of~$\Lambda_I$ which arrive respectively on the top and bottom endpoints of the $k$\ordinal{} commutator of~$\cN_\sq{Q}$. \fref{fig:flip} illustrates the properties of Proposition~\ref{prop:roots&flips} on the subword complex~$\subwordComplex[\bar{\sq{Q}}, \bar\rho]$ of Example~\ref{exm:toto}.
\end{example}


\section{The greedy flip tree}
\label{sec:greedyFlipTree}

\subsection{Increasing flips and greedy facets}
\label{subsec:greedyFacets}

Let~$I$ and~$J$ be two adjacent facets of~$\subwordComplex$ with~$I \ssm i = J \ssm j$. We say that the flip from~$I$ to~$J$ is \defn{increasing} if~$i<j$. This is equivalent to~$\Root{I}{i} \in \Phi^+$ by Proposition~\ref{prop:roots&flips}. We now consider the flip graph $\flipGraph$ oriented by increasing flips.

\begin{proposition}
\label{prop:increasingFlipGraph}
The graph~$\flipGraph$ of increasing flips is acyclic. The lexicographically smallest (resp.~largest) facet of~$\subwordComplex$ is the unique source (resp.~sink) of~$\flipGraph$.
\end{proposition}

\begin{proof}
The graph~$\flipGraph$ is acyclic, since it is a subgraph of the Hasse diagram of the order defined by~$I \le J$ iff there is a bijection~$\phi:I \to J$ such that~$i \le \phi(i)$ for all~$i \in I$. The lexicographically smallest facet is a source of~$\flipGraph$ since none of its flips can be decreasing. We prove that this source is unique by induction the word on~$\sq{Q}$. Denote by~$X(\sq{Q}_\eraseLast, \rho)$ (resp.~$X(\sq{Q}_\eraseLast, \rho q_m)$) the lexicographically smallest facet of~$\subwordComplex[\sq{Q}_\eraseLast, \rho]$ (resp.~$\subwordComplex[\sq{Q}_\eraseLast, \rho q_m]$) and assume that it is the unique source of the flip graph~$\flipGraph[\sq{Q}_\eraseLast, \rho]$ (resp.~$\flipGraph[\sq{Q}_\eraseLast, \rho q_m]$). Consider a source~$X$ of~$\flipGraph$. We distinguish two cases:
\begin{itemize}
\item If~$\ell(\rho q_m) > \ell(\rho)$, then~$q_m$ cannot be the last reflection of a reduced expression for~$\rho$. Thus~$\subwordComplex = \subwordComplex[\sq{Q}_\eraseLast, \rho] \join m$ and~$X = X(\sq{Q}_\eraseLast, \rho) \cup m$.
\item Otherwise,~$\ell(\rho q_m) < \ell(\rho)$. If~$m$ is in~$X$, it is flippable (by Proposition~\ref{prop:roots&flips}, since~$\Root{X}{m} = \rho(\alpha_{q_m}) \in \Phi^-\cap\rho(\Phi^+) = -\inv(\rho^{-1})$) and its flip is decreasing. This would contradict the assumption that~$X$ is a source of~$\flipGraph$. Consequently,~$m \notin X$. Since the facets of~$\subwordComplex$ which do not contain~$m$ coincide with the facets of~$\flipGraph[\sq{Q}_\eraseLast, \rho q_m]$, we obtain that $X = X(\sq{Q}_\eraseLast,\rho q_m)$.
\end{itemize}
In both cases, we obtain that the source~$X$ is the lexicographically smallest facet of~$\subwordComplex$.
The proof is similar for the sink.
\end{proof}

We call \defn{positive} (resp.~\defn{negative}) \defn{greedy facet} and denote by~$\positiveGreedy$ (resp.~$\negativeGreedy$) the unique source (resp.~sink) of the graph~$\flipGraph$ of increasing flips. The term ``positive'' (resp.~``negative'') emphasizes the fact that~$\positiveGreedy$ (resp.~$\negativeGreedy$) is the unique facet of~$\subwordComplex$ whose root configuration is a subset of positive (resp.~negative) roots, while the term ``greedy'' refers to the greedy properties of these facets underlined in Lemmas~\ref{lem:greedy1} and~\ref{lem:greedy2}.

\begin{example}
The greedy facets of the subword complex~$\subwordComplex[\bar{\sq{Q}},\bar\rho]$ of Example~\ref{exm:toto} are~$\positiveGreedy[\bar{\sq{Q}},\bar\rho] = \{1,2,3,5,6\}$ and~$\negativeGreedy[\bar{\sq{Q}},\bar\rho] = \{3,4,7,8,9\}$. See \fref{fig:greedy}.

\begin{figure}[h]
	\centerline{\includegraphics[width=.9\textwidth]{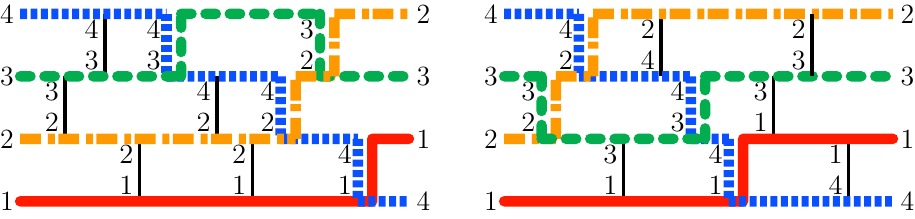}}
	\caption{The positive and negative greedy facets of~$\subwordComplex[\bar{\sq{Q}},\bar\rho]$.}
	\label{fig:greedy}
\end{figure}
\end{example}

The positive and negative greedy facets are clearly related by a reversing operation. More precisely, $\negativeGreedy[q_1 \cdots q_m,\rho] = \set{m+1-p}{p \in \positiveGreedy[q_m \cdots q_1,\rho^{-1}]}$. However, we will work in parallel with both positive and negative greedy facets, since certain results are simpler to understand and prove with~$\positiveGreedy$ while the others are simpler with~$\negativeGreedy$. In each proof, we only deal with the simplest situation and leave to the reader the translation to the opposite situation.

Remember that we denote by~$\sq{Q}_\eraseFirst \eqdef q_2 \cdots q_m$ and~$\sq{Q}_\eraseLast \eqdef q_1 \cdots q_{m-1}$ the words on~$S$ obtained from~$\sq{Q} \eqdef q_1 \cdots q_m$ by deleting its first and last letters respectively. We moreover denote by~$\shiftRight{X} \eqdef \set{x+1}{x \in X}$ and~$\shiftLeft{X} \eqdef \set{x-1}{x \in X}$ the right and left shifts of a subset~$X \subset \Z$. If~$\cX$ is a set of subsets of~$\Z$, we also write~$\shiftRight{\cX} \eqdef \set{\shiftRight{X}}{X \in \cX}$. Finally, remember that $\ell(\rho)$ denotes the length of~$\rho$, and that we write~${\rho \prec \sq{Q}}$ when~$\sq{Q}$ contains a reduced expression of~$\rho$.

The following two lemmas provide two (somehow inverse) greedy inductive procedures to construct the greedy facets~$\positiveGreedy$ and~$\negativeGreedy$. These lemmas are direct consequences of the definition of the greedy facets and the induction formulas~(\ref{eq:inductionRight}) and~(\ref{eq:inductionLeft}) on the subword complex.

\begin{lemma}
\label{lem:greedy1}
The greedy facets~$\positiveGreedy$ and~$\negativeGreedy$ can be constructed inductively from $\positiveGreedy[\varepsilon,e] = \negativeGreedy[\varepsilon,e] = \emptyset$ using the following formulas:
\begin{align*}
\positiveGreedy & = \begin{cases} \positiveGreedy[\sq{Q}_\eraseLast, \rho q_m] & \text{if } \ell(\rho q_m) < \ell(\rho), \\ \positiveGreedy[\sq{Q}_\eraseLast, \rho] \cup m & \text{otherwise.}\end{cases}
\\
\negativeGreedy & = \begin{cases} \shiftRight{\negativeGreedy[\sq{Q}_\eraseFirst, q_1\rho]} & \text{if } \ell(q_1\rho) < \ell(\rho), \\ 1 \cup \shiftRight{\negativeGreedy[\sq{Q}_\eraseFirst, \rho]} & \text{otherwise.}\end{cases}
\end{align*}
\end{lemma}

\begin{lemma}
\label{lem:greedy2}
The greedy facets~$\positiveGreedy$ and~$\negativeGreedy$ can be constructed inductively from $\positiveGreedy[\varepsilon,e] = \negativeGreedy[\varepsilon,e] = \emptyset$ using the following formulas:
\begin{align*}
\positiveGreedy & = \begin{cases} 1 \cup \shiftRight{\positiveGreedy[\sq{Q}_\eraseFirst,\rho]} & \text{if } \rho \prec \sq{Q}_\eraseFirst, \\ \shiftRight{\positiveGreedy[\sq{Q}_\eraseFirst,q_1\rho]} & \text{otherwise.} \end{cases}
\\
\negativeGreedy & = \begin{cases} \negativeGreedy[\sq{Q}_\eraseLast,\rho] \cup m & \text{if } \rho \prec \sq{Q}_\eraseLast, \\ \negativeGreedy[\sq{Q}_\eraseLast,\rho q_m] & \text{otherwise.} \end{cases}
\end{align*}
\end{lemma}

Lemmas~\ref{lem:greedy1} and~\ref{lem:greedy2} can be reformulated to obtain greedy sweep procedures on the word~$\sq{Q}$ itself, avoiding the use of induction. Namely, the positive greedy facet is obtained: 
\begin{enumerate}
\item either sweeping~$\sq{Q}$ from right to left placing inversions as soon as possible,
\item or sweeping~$\sq{Q}$ from left to right placing non-inversions as long as possible.
\end{enumerate}
The negative greedy facet is obtained similarly, inversing the directions of the~sweeps.


\subsection{The greedy flip tree}
\label{subsec:greedyFlipTree}

We construct in this section the positive and negative greedy flip trees of~$\subwordComplex$. This construction mainly relies on the following \defn{greedy flip property} of greedy facets.

\begin{proposition}
\label{prop:gfp}
If~$m$ is a flippable element of~$\negativeGreedy$, then~$\negativeGreedy[\sq{Q}_\eraseLast,\rho q_m]$ is obtained from~$\negativeGreedy$ by flipping~$m$. If~$1$ is a flippable element of~$\positiveGreedy$, then $\positiveGreedy[\sq{Q}_\eraseFirst,q_1\rho]$ is obtained from~$\positiveGreedy$ by flipping~$1$ and shifting to the left.
\end{proposition}

\begin{proof}
Although the formulation is simpler for the negative greedy facets, the proof is simpler for the positive ones (due to the direction chosen in the definition of the root function). Assume that~$1$ is a flippable element of~$\positiveGreedy$. Let~$J \in \facets$ and~$j \in J$ be such that $\positiveGreedy \ssm 1 = J \ssm j$. Consider the facet~$\shiftLeft{J}$ of~$\subwordComplex[\sq{Q}_\eraseFirst,q_1\rho]$ obtained shifting~$J$ to the left. Proposition~\ref{prop:roots&flips} (\ref{prop:roots&flips:update}) enables us to compute the root function~$\Root{J}{\cdot}$ for~$J$, which in turn gives us the root function for~$\shiftLeft{J}$:
$$\Root{\shiftLeft{J}}{k} = \begin{cases} \Root{\positiveGreedy}{k+1} & \text{if } 1 \le k \le j-1, \\ q_1(\Root{\positiveGreedy}{k+1}) & \text{otherwise.} \end{cases}$$
Since all positions~$i \in \positiveGreedy$ such that~$\Root{\positiveGreedy}{i} = \alpha_{q_1}$ are located before~$j$, and since $\alpha_{q_1}$ is the only positive root sent to a negative root by the simple reflection~$q_1$, all roots~$\Root{\shiftLeft{J}}{k}$, for~$k \in \shiftLeft{J}$, are positive. Consequently,~$\shiftLeft{J} = \positiveGreedy[\sq{Q}_\eraseFirst,q_1\rho]$.
\end{proof}

\begin{example}
Consider the subword complex of Example~\ref{exm:toto}. Since~$9$ is flippable in~$\negativeGreedy[\bar{\sq{Q}},\bar\rho] = \{3,4,7,8,9\}$, we have~$\negativeGreedy[\bar{\sq{Q}}_\eraseLast,\bar\rho\tau_1] = \{3,4,6,7,8\}$. Since~$1$~is flippable in~$\positiveGreedy[\bar{\sq{Q}},\bar\rho] = \{1,2,3,5,6\}$, we have~$\positiveGreedy[\bar{\sq{Q}}_\eraseFirst,\tau_2\bar\rho] = \shiftLeft{\{2,3,5,6,8\}} = \{1,2,4,5,7\}$.
\end{example}

We now define inductively the \defn{negative greedy flip tree}~$\negativeGreedyTree$. The induction follows the right induction formula~(\ref{eq:inductionRight}) for the facets~$\facets$. For the empty word~$\varepsilon$ and the identity~$e$ of~$W$, the tree~$\negativeGreedyTree[\varepsilon,e]$ is formed by the unique facet~$\emptyset$ of~$\subwordComplex[\varepsilon,e]$. For a non-empty word~$\sq{Q}$, we define the tree~$\negativeGreedyTree$ as
\begin{enumerate}[(i)]
\item $\negativeGreedyTree[\sq{Q}_\eraseLast, \rho q_m]$ if~$m$ appears in none of the facets of~$\subwordComplex$;
\item $\negativeGreedyTree[\sq{Q}_\eraseLast, \rho] \join m$ if~$m$ appears in all the facets of~$\subwordComplex$;
\item the disjoint union of~$\negativeGreedyTree[\sq{Q}_\eraseLast, \rho q_m]$ and~$\negativeGreedyTree[\sq{Q}_\eraseLast, \rho] \join m$, with an additional arc from~$\negativeGreedy[\sq{Q}_\eraseLast, \rho q_m]$ to~$\negativeGreedy = \negativeGreedy[\sq{Q}_\eraseLast, \rho] \cup m$, otherwise.
\end{enumerate}
See \fref{fig:negativeGreedyTree} for the negative greedy flip tree~$\negativeGreedyTree[\bar{\sq{Q}},\bar \rho]$ of Example~\ref{exm:toto}.

\begin{figure}[h]
	\centerline{
	\tikzstyle{every node}=[draw]
	\begin{tikzpicture}
		[
			level distance = 1cm,
			level 1/.style={sibling distance=3.5cm},
		    level 2/.style={sibling distance=2.1cm},
		    level 3/.style={sibling distance=1.4cm},
		]
		\node {$34789|$}
			child { node {$34678|$} 
				child { node {$13467|$}
					child { node {$13456|$}
						child { node {$123|56$} }
					}
					child { node {$123|67$} }
				}
				child { node {$3456|8$}
					child { node {$23|568$} }
				}
				child { node {$23|678$} }
			}
			child { node {$1347|9$} 
				child { node {$123|79$} }
			}
			child { node {$23|789$} }
		;
	\end{tikzpicture}}
	\caption{The negative greedy flip tree~$\negativeGreedyTree[\bar{\sq{Q}},\bar\rho]$ on the subword complex of Example~\ref{exm:toto}. Each facet~$I$ is denoted as the concatenation of its elements. The symbol~$|$ is explained in Example~\ref{exm:greedyIndex}.}
	\label{fig:negativeGreedyTree}
\end{figure}

We define similarly the \defn{positive greedy flip tree}~$\positiveGreedyTree$ of~$\subwordComplex$. The induction now follows the left induction formula~(\ref{eq:inductionLeft}) for the facets~$\facets$. The tree~$\positiveGreedyTree[\varepsilon,e]$ is formed by the unique facet~$\emptyset$ of~$\subwordComplex[\varepsilon,e]$. For a non-empty word~$\sq{Q}$, we define the tree~$\positiveGreedyTree$ as
\begin{enumerate}[(i)]
\item $\shiftRight{\positiveGreedyTree[\sq{Q}_\eraseFirst, q_1 \rho]}$ if~$1$ appears in none of the facets of~$\subwordComplex$;
\item $1 \join \shiftRight{\positiveGreedyTree[\sq{Q}_\eraseFirst, \rho]}$ if~$1$ appears in all the facets of~$\subwordComplex$;
\item the disjoint union of~$\shiftRight{\positiveGreedyTree[\sq{Q}_\eraseFirst, q_1 \rho]}$ and~$1 \join \shiftRight{\positiveGreedyTree[\sq{Q}_\eraseFirst, \rho]}$, with an additional arc from~$\positiveGreedy = 1 \cup \shiftRight{\positiveGreedy[\sq{Q}_\eraseFirst, \rho]}$ to~$\shiftRight{\positiveGreedy[\sq{Q}_\eraseFirst, q_1\rho]}$, otherwise.
\end{enumerate}
See \fref{fig:positiveGreedyTree} for the positive greedy flip tree~$\positiveGreedyTree[\bar{\sq{Q}},\bar \rho]$ of Example~\ref{exm:toto}.

\begin{figure}
	\centerline{
	\tikzstyle{every node}=[draw]
	\begin{tikzpicture}
		[
			grow'=up,
			level distance = 1cm,
			level 1/.style={sibling distance=2.3cm},
		    level 2/.style={sibling distance=1.5cm},
		]
		\node {$|12356$}
			child { node {$|23568$}
				child { node {$|34568$}
					child { node {$34|678$}
						child { node {$34|789$} }
					}
				}
				child { node {$23|678$}
					child { node {$23|789$} }
				}
			}
			child { node {$1|3456$}
				child { node {$134|67$}
					child { node {$134|79$} }
				}
			}
			child { node {$123|67$}
				child { node {$123|79$} }
			}
		;
	\end{tikzpicture}}
	\caption{The positive greedy flip tree~$\positiveGreedyTree[\bar{\sq{Q}},\bar\rho]$ on the subword complex of Example~\ref{exm:toto}. Each facet~$I$ is denoted as the concatenation of its elements. The symbol~$|$ is explained in Example~\ref{exm:greedyIndex}.}
	\label{fig:positiveGreedyTree}
\end{figure}

\begin{lemma}
\label{lem:greedyTree}
The negative (resp.~positive) greedy flip tree is a spanning trees of the increasing flip graph~$\flipGraph$, oriented towards its root~$\negativeGreedy$ (resp.~from its root~$\positiveGreedy$).
\end{lemma}

\begin{proof}
We prove the result for~$\negativeGreedyTree$ by induction on the length of~$\sq{Q}$. On the one hand, both the increasing flip graphs~$\flipGraph[\sq{Q}_\eraseLast, \rho q_m]$ and~$\flipGraph[\sq{Q}_\eraseLast, \rho] \join m$ are subgraphs of the increasing flip graph~$\flipGraph$. On the other hand, in the case where~$m$ appears in some but not all facets of~$\subwordComplex$, the additional arc from~$\negativeGreedy[\sq{Q}_\eraseLast, \rho q_m]$ to~$\negativeGreedy$ is an increasing flip according to Proposition~\ref{prop:gfp}.
\end{proof}

\begin{example}
Consider the subword complex~$\subwordComplex[\bar{\sq{Q}},\bar\rho]$ of Example~\ref{exm:toto}. Figures~\ref{fig:negativeGreedyTree} and \ref{fig:positiveGreedyTree} represent respectively the negative and the positive greedy flip trees~$\negativeGreedyTree[\bar{\sq{Q}},\bar\rho]$ and~$\positiveGreedyTree[\bar{\sq{Q}},\bar\rho]$. These trees are also represented on \fref{fig:flipGraph2} as spanning trees of the flip graph~$\flipGraph[\bar{\sq{Q}},\bar\rho]$ of \fref{fig:flipGraph}.
\begin{figure}[h]
	\centerline{\includegraphics[width=\textwidth]{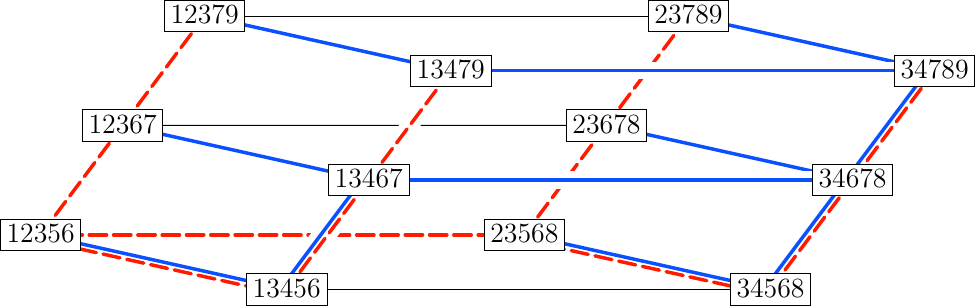}}
	\caption{The negative greedy flip tree~$\negativeGreedyTree[\bar{\sq{Q}},\bar\rho]$ (strong blue edges) and the positive greedy flip tree~$\positiveGreedyTree[\bar{\sq{Q}}, \bar\rho]$ (dashed red edges) are oriented spanning trees of the increasing flip graph~$\flipGraph[\bar{\sq{Q}},\bar\rho]$.}
	\label{fig:flipGraph2}
\end{figure}
\end{example}

The goal of the end of this section is to give a direct description of the greedy flip trees~$\negativeGreedyTree$ and~$\positiveGreedyTree$, avoiding the use of induction. Let~$I$ be a facet of the subword complex~$\subwordComplex$. We define the \defn{negative greedy index}~$\negativeGreedyIndex(I)$ of the facet~$I$ to be the last position~$x \in [m]$ such that $I \cap [x] = \negativeGreedy[q_1 \cdots q_x, \sigma_{[x] \ssm I}]$. In other words, the facet~$I$ is greedy until~$\negativeGreedyIndex(I)$ and not afterwards. Note in particular that~$I \cap [x]$ is greedy if and only if~$x \le \negativeGreedyIndex(I)$. Similarly, we define the \defn{positive greedy index}~$\positiveGreedyIndex(I)$ of the facet~$I$ to be the smallest position ${x \in [m]}$ such that ${\set{i-x}{i \in I \ssm [x]} = \positiveGreedy[q_{x+1} \cdots q_m, \sigma_{[x+1,m] \ssm I}]}$.

\begin{example}
\label{exm:greedyIndex}
Consider the subword complex~$\subwordComplex[\bar{\sq{Q}},\bar\rho]$ of Example~\ref{exm:toto}. The pseudoline arrangements associated to the facets~$\bar I \eqdef \{1,3,4,7,9\}$ and~$\bar J \eqdef \{3,4,7,8,9\}$ are represented in \fref{fig:flip}. We have~$\negativeGreedyIndex(\bar I) = 7$, while~$\negativeGreedyIndex(\bar J) = 9$ (\ie $\bar J$ is the negative greedy facet).

In \fref{fig:negativeGreedyTree}, the symbol~$|$ separates the elements smaller or equal to~$\negativeGreedyIndex(I)$ from those which are strictly larger than~$\negativeGreedyIndex(I)$. Similarly, in \fref{fig:positiveGreedyTree}, the symbol~$|$ separates the elements strictly smaller than~$\positiveGreedyIndex(I)$ from those which are larger or equal to~$\positiveGreedyIndex(I)$.
\end{example}

The following lemma provides the rule to update the greedy indices when we perform certain specific flips.

\begin{lemma}
\label{lem:greedyIndex}
Let~$I$ and~$J$ be two adjacent facets of~$\subwordComplex$ with~$I \ssm i = J \ssm j$. If $i < j \le \negativeGreedyIndex(J)$, then $\negativeGreedyIndex(I) = j-1$. If $\positiveGreedyIndex(I) \le i < j$, then $\positiveGreedyIndex(J) = i+1$.
\end{lemma}

\begin{proof}
We prove the result for the negative greedy index. On the one hand, we have $j \in J \cap [j] = \negativeGreedy[q_1 \cdots q_j,\sigma_{[j] \ssm J}] = \negativeGreedy[q_1 \cdots q_j,\sigma_{[j] \ssm I}]$. Since $j \notin I \cap [j]$, this implies that~$\negativeGreedyIndex(I) < j$.
On the other hand, the negative greedy flip property of Proposition~\ref{prop:gfp} ensures that~$I \cap [j-1] = \negativeGreedy[q_1 \cdots q_{j-1}, \sigma_{[j-1] \ssm I}]$ since it is obtained from~$J \cap [j] = \negativeGreedy[q_1 \cdots q_j,\sigma_{[j] \ssm J}]$ by flipping~$j$. Thus, $\negativeGreedyIndex(I) \ge j-1$.
\end{proof}

\begin{proposition}
\label{prop:greedyTree}
The negative greedy flip tree~$\negativeGreedyTree$ (resp. positive greedy flip tree~$\positiveGreedyTree$) has one vertex for each facet of~$\facets$, and one arc from a facet~$I$ to a facet~$J$ if and only if~$I \ssm i = J \ssm j$ for some~$i \in I$ and~$j \in J$ satisfying~$i < j \le \negativeGreedyIndex(J)$ (resp.~$\positiveGreedyIndex(I) \le i < j$).
\end{proposition}

\begin{proof}
We prove the result for the negative greedy flip tree by induction on the length of~$\sq{Q}$. We write here~$\negativeGreedyIndex_{\sq{Q},\rho}(I)$ to specify that we consider the negative greedy index of a set~$I$ regarded as a facet of~$\subwordComplex$. The result holds on the subword complex~$\subwordComplex[\varepsilon, e]$. Consider now two facets~$I$ and~$J$ of~$\subwordComplex$ with~$I \ssm i = J \ssm j$ for some $i < j$. We have three cases:
\begin{enumerate}[(i)]
\item If~$m$ is neither in~$I$ nor in~$J$, then~$I$ and~$J$ are both facets of~$\subwordComplex[\sq{Q}_\eraseLast,\rho q_m]$ and $\negativeGreedyIndex_{\sq{Q}_\eraseLast,\rho q_m}(J) = \min(\negativeGreedyIndex_{\sq{Q},\rho}(J), m-1)$. We thus conclude by induction.
\item If~$m$ is both in~$I$ and in~$J$, then~$I \ssm m$ and~$J \ssm m$ are both facets of~$\subwordComplex[\sq{Q}_\eraseLast,\rho]$ and $\negativeGreedyIndex_{\sq{Q}_\eraseLast,\rho}(J \ssm m) = \min(\negativeGreedyIndex_{\sq{Q},\rho}(J), m-1)$. We thus conclude by induction.
\item Otherwise, $m$ is in precisely one of the facets~$I$ and~$J$. Thus, we must have~${j = m}$. If~$j \le \negativeGreedyIndex(J)$, then $J = \negativeGreedy$, $I = \negativeGreedy[\sq{Q}_\eraseLast,\rho q_m]$, and the flip from~$I$ to~$J$ is an arc of~$\negativeGreedyTree$. Conversely, if~$j > \negativeGreedyIndex(J)$, then $J \ne \negativeGreedy$ and the flip from~$I$ to~$J$ is not an arc of~$\negativeGreedyTree$. \qedhere
\end{enumerate}
\end{proof}

Although we defined the greedy flip trees~$\negativeGreedyTree$ and~$\positiveGreedyTree$ inductively, the results of Lemma~\ref{lem:greedyIndex} and Proposition~\ref{prop:greedyTree} enable us to construct them directly on the graph~$\flipGraph$, avoiding the use of induction. We use this construction to provide a non-inductive enumeration scheme for the facets of~$\subwordComplex$.


\subsection{The greedy flip algorithm}
\label{subsec:greedyFlipTree}

The \defn{greedy flip algorithm} generates all facets of the subword complex~$\subwordComplex$ by a depth first search procedure on the (positive or negative) greedy flip tree. The preorder traversal of the greedy flip tree also provides an iterator on the facets of~$\subwordComplex$. Given a facet~$I$ of~$\subwordComplex$, we can indeed compute its next element in the preorder traversal, provided we know its root function, its greedy index and the path from~$I$ to the root in the greedy flip tree. These data can be updated at each step of the algorithm, using Proposition~\ref{prop:roots&flips} for the root function and Lemma~\ref{lem:greedyIndex} for the greedy index.

To evaluate the running time and working space of the greedy flip algorithm, remember that we consider as parameters both the rank~$n$ of the group~$W$ and the size~$m$ of the word~$\sq{Q}$. During the algorithm, we only need to remember the current facet, together with its root function, its greedy index, and its path to the root in the greedy flip tree. Thus, the working space of the algorithm is in~$O(mn)$. Concerning running time, each facet needs at most~$m$ flips to generate all its children in the greedy flip tree. Since a flip can be performed in~$O(mn)$ time (see Section~\ref{subsec:roots&flips}), the running time per facet of the greedy flip algorithm is in~$O(m^2n)$.

\begin{figure}[b]
	\centerline{\includegraphics[width=.5\textwidth]{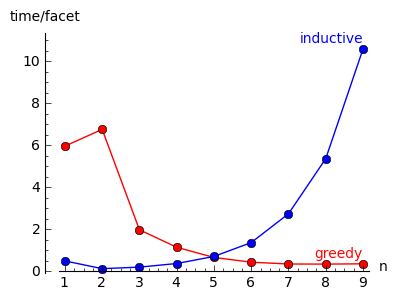} \quad \includegraphics[width=.5\textwidth]{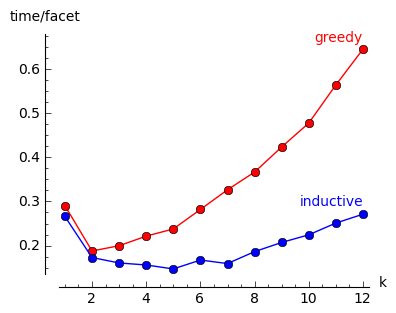}}
	\caption{Comparison of the running times of the inductive algorithm and the greedy flip algorithm to generate the $k$-cluster complex of type~$A_n$. On the left, $k$ is fixed at~$1$ while~$n$ increases; on the right, $n$ is fixed at~$3$ while~$k$ increases. The time is presented in millisecond per facet.}
	\label{fig:runnngTimes}
\end{figure}

We have implemented the greedy flip algorithm using the mathematical software Sage~\cite{sage}. This implementation is integrated into C.~Stump's patch on subword complexes. The user can now select either the inductive algorithm (directly based on the inductive structure of the subword complex as discussed in Section~\ref{subsec:generationgSubwordComplex}) or the greedy flip algorithm. We have seen that these two algorithms have the same theoretical complexity. To compare their experimental running time, we have constructed the $k$-cluster complex of type~$A_n$ for increasing values of~$k$ and~$n$. Its facets correspond to the $k$-triangulations of the $(n+2k+1)$-gon (see Example~\ref{exm:geometricGraphs} and~\cite{CeballosLabbeStump} for the definition of multicluster complexes in any finite type). The rank of the group is~$n$, while the length of the word is~$kn+{n \choose 2}$. \fref{fig:runnngTimes} presents the running time per facet for both enumeration algorithms in two situations: on the left, $k$ is fixed at~$1$ while~$n$ increases; on the right, $n$ is fixed at~$3$ while~$k$ increases. The greedy flip algorithm is better than the inductive algorithm in the first situation, and worst in the second. We observe a similar behavior for the computation of $k$-cluster complexes of types~$B_n$ and~$D_n$. In general, the inductive algorithm is experimentally faster when the Coxeter group is fixed, but slower when the size of the Coxeter group increases.

\begin{remark}
For type~$A$ spherical subword complexes, our algorithm is similar to that of~\cite{PilaudPocchiola} (which was formulated in terms of primitive sorting networks). Observe however that, contrarily to~\cite{PilaudPocchiola}, we allow~$\rho$ to be any element of~$W$. This slight generalization enables us to provide an inductive definition for the greedy flip tree, which simplifies the presentation of the algorithm. For the subword complexes which provide combinatorial models for pointed pseudotriangulations (see Example~\ref{exm:geometricGraphs}), our algorithm coincides with the greedy flip algorithm of~\cite{BronnimannKettnerPocchiolaSnoeying}.
\end{remark}


\section*{Acknowledgments}

I am grateful to C.~Stump for fruitful discussions on subword complexes and related topics and to M.~Pocchiola for introducing me to the greedy flip algorithm on pseudotriangulations. I also thank three anonymous referees for valuable comments and suggestions on this paper, in particular for pointing out an important mistake in a previous version of this paper. I thank the Sage and Sage-Combinat development team for making available this powerful mathematics software and C.~Stump again for helpful support on the Sage implementation of Coxeter groups and subword complexes.

\bibliographystyle{alpha}
\bibliography{greedyFlip.bib}

\end{document}